\documentclass[11pt,letterpaper]{amsart}
\usepackage{amsmath, enumerate}
\usepackage{amsfonts}
\usepackage{amssymb}
\usepackage{amsthm}
\usepackage{graphicx}
\usepackage{soul}
\usepackage[usenames, dvipsnames]{color} 
\usepackage{hyperref}

\newcommand{\measurerestr}{%
  \,\raisebox{-.127ex}{\reflectbox{\rotatebox[origin=br]{-90}{$\lnot$}}}\,%
}

\usepackage{tikz}
\newcommand*\circled[1]{\tikz[baseline=(char.base)]{
            \node[shape=circle,draw,inner sep=2pt] (char) {#1};}}

\newtheorem{Theorem}{Theorem}[section]
\newtheorem{Corollary}[Theorem]{Corollary}
\newtheorem{Lemma}[Theorem]{Lemma}
\newtheorem{Proposition}[Theorem]{Proposition}
\theoremstyle{definition}
\newtheorem{Definition}[Theorem]{Definition}

\theoremstyle{remark}
\newtheorem{Remark}[Theorem]{Remark}

\def \O {\Omega}
\def \bO {\partial \Omega}
\def \R {\mathbb{R}}
\def \Rn {\mathbb{R}^n}
\def \S{\mathbb{S}}

\def \e {\epsilon}

\def \a {\alpha}

\def \t {\tau}

\def \dist{\mathrm{dist}}

\def \tr{\mathrm{tr}}
\def \det{\mathrm{det}}

\def \Id{\mathrm{Id}}

\newcommand{\norm}[1]{\lvert #1\rvert}
\newcommand{\Norm}[1]{\lVert #1\rVert}

\newcommand{\inner}[2]{\langle #1, #2\rangle}
\def\hstar{h^*}

\def\G{G}

\DeclareMathOperator{\MTW}{MTW}
\def\tmax{t_{\max}}
\def\MTWconst{\sigma_{\text{MTW}}}
\makeatletter
\@namedef{subjclassname@2020}{%
  \textup{2020} Mathematics Subject Classification}
\makeatother

\begin{document}

\title[Perturbative approach to parabolic optimal transport]{A perturbative approach to the parabolic optimal transport problem for non-MTW costs}
\author[F. Abedin and J. Kitagawa]{Farhan Abedin$^{*}$ and Jun Kitagawa$^{**}$} 
\address{Department of Mathematics, University of Utah, Salt Lake City, UT 84112}
\email{abedinf@math.utah.edu}
\address{Department of Mathematics, Michigan State University, East Lansing, MI 48824}
\email{kitagawa@math.msu.edu}
\thanks{$^{*}$FA acknowledges support through an AMS--Simons Travel Grant for the period 2020--2022. \\ $^{**}$JK's research was supported in part by National Science Foundation grant DMS-2000128.}
\begin{abstract}
Fix a pair of smooth source and target densities $\rho$ and $\rho^*$ of equal mass, supported on bounded domains $\O,\O^* \subset \Rn$. Also fix a cost function $c_0 \in C^{4,\a}(\overline{\O} \times \overline{\O^*})$ satisfying the weak regularity criterion of Ma, Trudinger, and Wang, and assume $\O$ and $\O^*$ are uniformly $c_0$- and $c_0^*$-convex with respect to each other. We consider a parabolic version of the optimal transport problem between $(\O,\rho)$ and $(\O^*,\rho^*)$ when the cost function $c$ is a sufficiently small $C^4$ perturbation of $c_0$, and where the size of the perturbation depends on the given data. Our main result establishes global-in-time existence of a solution $u \in C^2_xC^1_t(\overline\O \times [0, \infty))$ of this parabolic problem, and convergence of $u(\cdot,t)$ as $t \to \infty$ to a Kantorovich potential for the optimal transport map between $(\O,\rho)$ and $(\O^*,\rho^*)$ with cost function $c$. A noteworthy aspect of our work is that $c$ does \emph{not} necessarily satisfy the weak Ma-Trudinger-Wang condition.
\end{abstract}

\maketitle

\section{Introduction}\label{sec:intro}

The optimal transport problem is intimately tied to the theory of  second-order elliptic equations via a fully nonlinear PDE of Monge-Amp\`ere type coupled with the so-called second boundary condition \cite{MaTrudingerWang05}. This connection creates a natural method for proving the existence of optimal maps: solve a parabolic version of this PDE whose stationary state is a solution of the elliptic problem, and let the time variable tend to infinity. We refer to this as the \emph{parabolic optimal transport problem}; see \eqref{PDE} in Section \ref{sec: preliminaries} for a precise formulation. Such an asymptotic approach has been shown to work in the papers \cite{KimStreetsWarren12, Kitagawa12, AbedinKitagawa20, Schnurer03}, and provides a natural algorithm for approximating optimal maps owing to its fast rate of convergence to the steady state \cite{KimStreetsWarren12, AbedinKitagawa20}.

Existing results for the parabolic optimal transport problem make crucial use of the so-called weak Ma-Trudinger-Wang (MTW) condition (see Definition \ref{def: MTW}). This is a sign condition on a fourth-order tensor related to the cost function that is necessary for the regularity of optimal maps even in the stationary case \cite{Loeper09}. Due to the fact that the weak-MTW condition is difficult to verify and is not satisfied by many cost functions, it is a natural question to determine if, given fixed source and target measures, one can solve the parabolic optimal transport problem when the cost function fails to satisfy \eqref{A3} by a quantifiable amount. Our results in this paper show that, given a fixed pair of source and target measures and a cost function $c_0$ which satisfies \eqref{A3}, the solution to the parabolic problem \eqref{PDE} exists for all time and converges to a solution of the elliptic problem when the cost function $c$ is a sufficiently small $C^4$ perturbation of $c_0$; the smallness of this perturbation depends on the source and target measures and other structural quantities. We stress that $c$ does \emph{not} necessarily satisfy \eqref{A3}. 

Our main theorem is as follows. We refer to Section \ref{sec: preliminaries} for precise definitions and terminology.

\begin{Theorem}\label{thm : main} 
Let $\O, \O^* \subset \Rn$ be smooth, bounded domains and let $\rho(x) \ dx$, $\rho^*(y) \ dy$, be absolutely continuous measures supported on $\O, \O^*$ respectively satisfying  \eqref{densitybounds} and \eqref{massbalance}. Suppose there exists $r_0 > 0$ such that the cost function $c_0 \in C^{4,\a}(\overline{\O} \times \overline{\mathcal{N}_{r_0}(\O^*)})$ for some $\a \in (0,1]$ satisfies the conditions \eqref{bi-twist}, \eqref{mixedHessianinvertible}, and \eqref{A3}. Assume, in addition that the domains $\O, \O^*$  satisfy \eqref{cconvexity} and \eqref{cstarconvexity} with respect to $c_0$. Then there exists a constant $\hat{R} > 0$ depending only the structure (see Definition \ref{def: structural constants}) such that if $c \in C^{4,\a}(\overline{\O} \times \overline{\mathcal{N}_{r_0}(\O^*)})$ satisfies $||c - c_0||_{C^4(\overline{\O} \times \overline{\mathcal{N}_{r_0}(\O^*)})} \leq \hat{R}$, then there exists a locally uniformly $c$-convex function $u_{\text{initial}} \in C^{4,\a}(\O)$ satisfying \eqref{boundarycompatbility} such that a  solution $u\in C^2_x C^1_t(\overline{\O}\times [0, \infty))$ of the flow \eqref{PDE} exists for all time and $u(\cdot,t)$ converges in $C^2(\overline\O)$ as $t \to \infty$ to a $c$-convex function $u^{\infty}$ which is a Kantorovich potential for the optimal transport problem between $(\rho, \O)$ and $(\rho^*, \O^*)$ with cost $c$.
\end{Theorem}

\begin{Remark} We can actually obtain exponential convergence of the flow from our previous result \cite{AbedinKitagawa20}, where the \eqref{A3} condition was not used.
\end{Remark}

\subsection{A Key Example}

Let $c_0(x,y):= \frac{1}{2}|x-y|^2$. Assume $\O,\O^*$ are fixed uniformly convex domains such that $\dist(\O, \O^*) > 0$, and let $\rho, \rho^*$ be fixed measures supported on $\O,\O^*$ respectively, and satisfying \eqref{densitybounds} and \eqref{massbalance}. For $p \in [-2,2] \setminus \{0\}$, consider the cost $c(x,y) := \frac{1}{p}|x - y|^p$. It is known that $c$ satisfies \eqref{bi-twist} and \eqref{mixedHessianinvertible} for $p \in [-2,2] \setminus \{0,1\}$ (as long as $\dist(\O, \O^*) > 0$) but fails to satisfy \eqref{A3} for any $p \notin [-2,1) \cup \{2\}$; see \cite[Section 8, Example 4]{TrudingerWang09}. As a consequence of Theorem \ref{thm : main}, we have the following result:

\begin{Proposition} Let $c_0, \O, \O^*, \rho, \rho^*$ be as above. Suppose there exist constants $M_1, M_2 > 0$ such that
$$0 < M_1 \leq |x-y| \leq M_2 \text{ for all } (x,y) \in \O \times \O^*.$$
Then there exist $\gamma_0, r_0 > 0$ depending only on $M_1, M_2$ and the structure of the problem such that if $|p \pm 2| \leq \gamma_0$, then $||c - c_0||_{C^4(\overline{\O} \times \overline{\mathcal{N}_{r_0}(\O^*)})} \leq \hat{R}$. Consequently, there is an initial condition for which the solution of the flow \eqref{PDE} corresponding to $c$ exists for all time and converges in $C^2$ to a solution of the optimal transport problem between $(\rho, \O)$ and $(\rho^*, \O^*)$ for the $p$-th power cost with $|p \pm 2| \leq \gamma_0$. \end{Proposition}

This is a parabolic counterpart (when perturbing around $p=2$) of results obtained previously by  Caffarelli-Gonz\'alez-Nguyen  \cite{CaffarelliGonzalezNguyen14} and Chen-Figalli \cite{ChenFigalli15} in the elliptic case under slightly different assumptions on the cost and densities. We refer the reader to Subsection \ref{relatedliterature} for more details on these and related works.

\subsection{Remarks on the main theorem and proof strategy}\label{remarksonmainthm}

As stated previously, the condition \eqref{A3} is not necessarily preserved by $C^4$ perturbations. By a counterexample by Loeper in  \cite{Loeper09}, it is known the condition is sharp in the sense that if a cost function fails to satisfy \eqref{A3}, then there exist smooth source and target measures for which the optimal transport map is discontinuous. However, since the threshold value $\hat{R}$ depends on the structure of the problem (in particular on the particular choices of $\rho$ and $\rho^*$), this does not violate Loeper's counterexample. We also do not require any special structure on the source and target measures beyond \eqref{massbalance}, \eqref{densitybounds}, and sufficient smoothness.

Let us comment on the strategy behind the proof of Theorem \ref{thm : main}. We note that the conclusion of the theorem does not immediately follow from an inverse function theorem type argument. In fact, such an approach to proving Theorem \ref{thm : main} would require a version of the inverse function theorem on Frech\'et manifolds, which is known to be false in general. As such, we must revisit the approach taken in \cite{Kitagawa12} for proving long-time existence and convergence of the flow \eqref{PDE}, and taking care to not make use of the \eqref{A3} condition. A number of a priori estimates from \cite{Kitagawa12} hold true even in the non-MTW setting, but a missing essential ingredient is the global $C^2$ estimate. In \cite{Kitagawa12}, the interior $C^2$ estimate makes crucial use of \eqref{A3}, while the boundary $C^2$ estimate relies on the validity of an appropriate interior estimate. To circumvent the role played by \eqref{A3}, we quantify the failure of \eqref{A3} and establish a dichotomy (see Proposition \ref{prop : dichotomy}) that allows us to exclude blow-up of the $C^2$ norm of the solution if the initial data is chosen appropriately. Such a dichotomy argument relies on the $C^4$ closeness of the cost function to one that satisfies \eqref{A3}, and is in the spirit of \cite{Warren11}. The construction of appropriate initial data requires an implicit function theorem argument which also makes use of the $C^4$ closeness to an MTW cost.

\subsection{Comments on related literature}\label{relatedliterature}

\subsubsection{Perturbation results for elliptic Monge-Amp\`ere equations} Only a limited number of works address the regularity of solutions to elliptic Monge-Amp\`ere equations arising from optimal transport problems when the cost function does not satisfy \eqref{A3}. The first such result was obtained by Caffarelli, Gonz\'alez, and Nguyen \cite{CaffarelliGonzalezNguyen14}, who consider the cost function $c(x,y) = \frac{1}{p}|x-y|^p$ with $p \approx 2$ and domains $\O, \O^*$ that are uniformly convex and a positive distance apart. As mentioned previously, these cost functions do not satisfy \eqref{A3} unless $p = 2$. Global versions of the results in \cite{CaffarelliGonzalezNguyen14} were obtained by Chen and Figalli \cite{ChenFigalli15}. Their more recent paper \cite{ChenFigalli16} investigates the regularity of optimal maps when the cost function is a $C^2$ perturbation of an MTW cost. The proof strategy in all the aforementioned works uses localization arguments originally due to Caffarelli \cite{Caffarelli90b} and shows that the optimal map for the perturbed cost inherits regularity from the optimal map for the quadratic cost (or some other fixed cost satisfying \eqref{A3}). We note that our approach in the present paper, aside from being in the parabolic setting, differs from that in  \cite{CaffarelliGonzalezNguyen14, ChenFigalli15, ChenFigalli16} as we do not perform any localization arguments; however, we do require higher regularity of the cost function as well as the source and target measures. Finally, the work of Warren \cite{Warren11} establishes regularity results for optimal transport problems between log-concave mass distributions supported on small balls with cost functions that are sufficiently close to the quadratic cost. As mentioned in Subsection \ref{remarksonmainthm}, our dichotomy argument for establishing the parabolic $C^2$ estimates is inspired by the corresponding elliptic estimates in \cite{Warren11}, however we do not require log-concavity of the measures which is an essential ingredient in Warren's result.

\subsubsection{Perturbation of distance-squared cost on Riemannian manifolds} A different class of perturbative results exists in the optimal transport literature due to work of Figalli, Rifford, Villani, and Loeper \cite{LoeperVillani10, FigalliRifford09, FigalliRiffordVillani11, FigalliRiffordVillani12}. These papers consider $C^4$ perturbations of canonical metrics on Riemannian manifolds and the associated cost function given by the Riemannian distance squared. The qualitative difference between these results and ours is that the perturbed cost functions considered in \cite{LoeperVillani10, FigalliRifford09, FigalliRiffordVillani11, FigalliRiffordVillani12} end up satisfying a form of \eqref{A3}; in fact, this is one of the main contributions of the aforementioned papers. A separate collection of perturbative results is due to Delano\"e and Ge \cite{Delanoe04, DelanoeGe10, DelanoeGe11, Delanoe15}. These papers consider distance-square costs for Riemannian metrics whose Gauss curvature is almost constant, and prove regularity of the optimal map.

\subsubsection{Gradient flow of the cost functional} The flow \eqref{PDE} provides a natural algorithm for approximating optimal maps, but is one of only a handful of known asymptotic methods that have been proposed for solving the optimal transport problem \cite{AngenentHakerTannebaum03, FaccaCardinPutti18}. Among these, the one most relevant to the present paper originates in the work of Angenent, Haker, and Tannenbaum \cite{AngenentHakerTannebaum03}. While the approach in \cite{AngenentHakerTannebaum03} easily adapts to arbitrary cost functions and densities, it has been most studied in the case 
\begin{equation}\label{AHTconfiguration}
c(x,y) = |x-y|^2, \quad \rho = \mathcal{L}^n \measurerestr \O, \quad \rho^* \text{ a.c. w.r.t } \mathcal{L}^n, \  \text{supp}(\rho^*) = \O^*.
\end{equation}
The key idea in \cite{AngenentHakerTannebaum03} is to consider a one-parameter family of admissible maps $T^t$ such that the function
$$t \mapsto \mathcal{C}(T^t) := \int_{\O} |x-T^t(x)|^2 \ dx$$
is decreasing. This can be achieved if $T^t$ solves the nonlinear and non-local vectorial transport problem
\begin{equation}\label{AHT}\tag{AHT}
\begin{cases}
\dot{T^t} + (\mathbf{v}^t  \cdot \nabla) T^t = 0,\\
\mathbf{v}^t = \mathbb{P}T^t,\\
(T^0)_{\#}\rho = \rho^*,
\end{cases}
\end{equation}
where $\mathbb{P}$ is the so-called Leray projector onto divergence-free vector fields in $\O$ satisfying a no-flux condition on $\partial \O$. The short-time existence of the flow \eqref{AHT} and infinite-time existence of a regularized version is proved in \cite{AngenentHakerTannebaum03}; see \cite[Chapter 6, Section 6.2]{SantambrogioBook15} for  additional results in two dimensions. Note that any gradient is a fixed point of \eqref{AHT}. It is an interesting and challenging problem to determine conditions on $T^0$ guaranteeing the infinite-time existence and convergence of \eqref{AHT} to an admissible map that is also the gradient of a convex function; by a celebrated theorem of Brenier \cite{Brenier91}, such a map is the unique optimal transport map between $\rho$ and $\rho^*$ for the quadratic cost. Some progress towards solving this problem has been made in the recent paper \cite{NguyenNguyen20}, whose authors show that if $T^0$ is sufficiently close in $H^s(\O), \ s > 1+\frac{n}{2},$ to the gradient of a uniformly convex function $\varphi$ (where the closeness depends on $\varphi$), then $T^t$ converges exponentially fast in $H^{s-1}(\O)$  as $t \to \infty$ to the optimal map. An appealing aspect of the result in \cite{NguyenNguyen20} is that it applies to smooth domains $\O,\O^*$ that are not necessarily convex; note that, in this case, the optimal map need not even be continuous, owing to counterexamples constructed by Caffarelli \cite{Caffarelli92} (also see \cite{Jhaveri19}). On the other hand, the conditions imposed on $T^0$ are rather stringent; moreover, each target density $\rho^*$ will require an appropriate choice of $T^0$ to ensure the convergence of the flow. In contrast, the results in the present paper, combined with our previous works \cite{Kitagawa12, AbedinKitagawa20}, show that for the configuration \eqref{AHTconfiguration} with $\O, \O^*$ convex, the parabolic Monge-Amp\`ere-type equation \eqref{PDE} exhibits desirable asymptotic behavior when initialized using a single function $u_{\text{initial}}$ for \emph{any} target density $\rho^*$ supported on $\O^*$. The function $u_{\text{initial}}$ can also be constructed explicitly by solving the optimal transport problem for the configuration \eqref{AHTconfiguration} with $\rho^* := \frac{\text{Vol}(\O)}{\text{Vol}(\O^*)} \mathcal{L}^n \measurerestr \O^*$. This makes \eqref{PDE} arguably more flexible than the system \eqref{AHT} when the source and target domains are convex, considering the convergence criteria presently available for the latter flow.
 
\subsection{Additional Remarks}

\begin{Remark} 
The $C^4$ closeness of $c$ to $c_0$ in $\O \times \mathcal{N}_{r_0}(\O^*)$ will guarantee that $c$ satisfies \eqref{bi-twist}. However, in certain situations, \eqref{bi-twist} can be verified independently (for example,  in the case of the $p$-th power cost). For another class of such examples, let $c_0(x,y):= \frac{1}{2}|x-y|^2$. Consider the cost $c(x,y) = c_0(x,y) + \eta(x,y)$ for some smooth function $\eta$. Let us assume that $\eta$ satisfies the following \emph{anti-monotonicity} condition:
\begin{align}\label{anti-monotone}
y &\mapsto -\nabla_x \eta(x,y) \text{ is a monotone map }\quad \forall x\in \overline\O,\notag\\
x &\mapsto -\nabla_y \eta(x,y) \text{ is a monotone map } \quad \forall y\in \overline{\O^*}.
\end{align}
We claim if $\eta$ satisfies \eqref{anti-monotone}, then $c$ satisfies \eqref{bi-twist}. Indeed,
$$\nabla_x c(x,y) = x-y + \nabla_x \eta(x,y).$$
Consequently, if there exists $x_0 \in \O$ and $y_1, y_2 \in \O^*$ such that $\nabla_x c(x_0, y_1) = \nabla_x c(x_0, y_1)$, it follows that
$$y_2-y_1  = \nabla_x \eta(x_0,y_2) - \nabla_x \eta(x_0,y_1).$$
This implies
$$|y_2-y_1|^2  = \left\langle \nabla_x \eta(x_0,y_2) - \nabla_x \eta(x_0,y_1), y_2 - y_1 \right\rangle \leq 0,$$
where in the final inequality we have used \eqref{anti-monotone}. We thus conclude $y_2 = y_1$. Carrying out a similar argument in the $x$ variables, we conclude that $c$ satisfies \eqref{bi-twist} whenever $\eta$ satisfies \eqref{anti-monotone}. 
\end{Remark}

\begin{Remark}
The $C^4$ closeness of $c$ to $c_0$ also allows us to prove existence of appropriate initial conditions $u_{\text{initial}}$ for the flow. As the arguments in Section \ref{sec : initialdata} will show, it is possible to construct $u_{\text{initial}}$ by solving Poisson's equation with an oblique boundary condition. In addition, the initial condition for the cost $c$ will be close to that of $c_0$, and the initial condition for $c_0$ can be chosen to be the solution of the optimal transport problem for $c_0$ with constant densities; this function is guaranteed to be $C^{2,\a}$ smooth by regularity theory for cost functions satisfying \eqref{A3} \cite{TrudingerWang09}.
\end{Remark}

\subsection{Outline of the paper} The remainder of this paper is structured as follows. Section \ref{sec: preliminaries} sets up the necessary notation to precisely state the parabolic optimal transport problem and recalls a number of useful estimates from \cite{Kitagawa12}. Section \ref{sec : initialdata} is devoted to the construction of suitable initial data for the flow \eqref{PDE} when the cost function is sufficiently close to one satisfying \eqref{A3}. Section \ref{sec : c2estimate} establishes the crucial second derivative estimates necessary for the infinite-time existence and convergence to steady state of the flow \eqref{PDE}, which is proved in Section \ref{sec : mainthm}. Appendix \ref{polyprop} collects a number of useful facts about a certain $n$-th degree polynomial that will be used in the proof of the $C^2$ estimates in Section \ref{sec : c2estimate}.

\section{Preliminaries}\label{sec: preliminaries}
\subsection{Notation and Setup}\label{subsec:notation}
We will assume from here onward that $\O$, $\O^*$ are open, smooth, bounded domains in $\R^n$. The outward-pointing unit normal to $\partial \O$ (resp. $\partial \O^*$) will be denoted by $\nu$ (resp. $\nu^*$). The function $\hstar$ will be a $C^2$, normalized defining function for $\O^*$; i.e. $\hstar = 0$ on $\partial \O$, $\hstar < 0$ on $\O$, and $\nabla \hstar = \nu^*$ on $\partial \O^*$ (for existence of such a function, see \cite[Appendix A]{Kitagawa12}). The measures $\rho(x) dx, \rho^*(y) dy$ will be assumed to be absolutely continuous with respect to $n$-dimensional Lebesgue measure and satisfy
\begin{equation}\label{densitybounds}\tag{Den Bds}
0 < \lambda \leq \rho, \rho^* \leq \lambda^{-1} < \infty \quad \text{for a constant } \lambda, \text{ and }
\end{equation}
\begin{equation}\label{massbalance}\tag{Mass Bal}
\quad \int_{\O} \rho = \int_{\O^*} \rho^*.
\end{equation}
We will also assume $c \in C^{4,\a}(\overline{\O} \times \overline{\O^*})$ for some $\a \in (0,1]$, and satisfies the bi-twist conditions:
\begin{align}
y &\mapsto -\nabla_x c(x,y) \text{ is injective }\forall x\in \overline\O,\notag\\
x &\mapsto -\nabla_y c(x,y) \text{ is injective }\forall y\in \overline{\O^*}.\label{bi-twist}\tag{Bi-Twist}
\end{align}
In addition, we assume that 
\begin{equation}\label{mixedHessianinvertible}\tag{Non-Deg}
\det D^2_{x,y} c(x,y) \neq 0.
\end{equation}
For any $p \in -\nabla_x c(x,\O^*)$ and $x \in \O$, (resp. $q \in -\nabla_y c(\O, y)$ and $y \in \O^*$), we denote by $\exp^c_x(p)$ (resp. $\exp^{c^*}_y(q)$) the unique element of $\O^*$ (resp. $\O$) such that
\begin{equation}\label{A1}
-\nabla_x c(x,\exp^c_x(p))= p, \quad -\nabla_y c(\exp^{c^*}_y(q),y) = q.
\end{equation}

We say \emph{$\O$ is $c$-convex with respect to $\O^*$} if the set $-\nabla_y c(\O, y)$ is a convex set for each $y \in \O^*$. Similarly, \emph{$\O^*$ is $c^*$-convex with respect to $\O$} if the set $-\nabla_x c(x, \O^*)$ is a convex set for each $x \in \O$. Analytically, these conditions are satisfied if we have
\begin{equation}\label{cconvexity}\tag{Dom $c$-Conv}
\left[\nu^j_i(x) - c^{\ell, k} c_{ij,\ell}(x,y) \nu^k(x) \right] \t^i \t^j \geq \delta \norm{\tau}^2 \quad \forall \ x \in \bO, \ y \in \overline{\O^*}, \ \tau \in T_x(\bO)
\end{equation}
and
\begin{equation}\label{cstarconvexity}\tag{Tar $c^*$-Conv}
\left[(\nu^*)^j_i(y) - c^{k, \ell} c_{\ell, ij}(x,y) (\nu^*)^k(x) \right] (\t^*)^i (\t^*)^j \geq \delta^*\norm{\t^*}^2 \quad \forall \ y \in \bO^*, \ x \in \overline{\O}, \ \tau^* \in T_y(\bO^*)
\end{equation}
for some constants $\delta$, $\delta^*\geq 0$ respectively, where we will always sum over repeated indices. If $\delta$ (resp. $\delta^*$) is strictly positive, we say that \emph{$\O$ is uniformly $c$-convex with respect to $\O^*$} (resp. \emph{$\O^*$ is uniformly $c^*$-convex with respect to $\O$}).

For the remainder of the paper, we will fix a distinguished cost function $c_0$, along with domains $\O$ and $\O^*$ which are uniformly $c_0$- and $c^*_0$-convex with respect to each other. Additionally, we assume that $c_0$ satisfies conditions \eqref{bi-twist} and \eqref{mixedHessianinvertible} on $\overline \O\times \overline{\mathcal{N}_{r_0}(\O^*)}$ for some fixed radius $r_0>0$ where $r_0$ is small enough that $\mathcal{N}_{r_0}(\O^*)$ has a $C^1$ boundary; such an $r_0$ always exists due to regularity of the distance function of a smooth domain in a small neighborhood of the boundary.
\begin{Definition}\label{def: structural constants}
 We will say that a constant \emph{depends on the structure of the problem} if it depends only on the following quantities:
 \begin{itemize}
 \item $r_0$, $n$, $\text{diam}(\O)$, $\text{diam}(\O^*)$;
 \item $\Norm{\rho}_{C^2(\overline\O)}$, $\Norm{\rho^*}_{C^2(\overline{\O^*})}$, and the constant $\lambda$  in \eqref{densitybounds};
\item $\Norm{c_0}_{C^4(\overline\O \times \overline{\mathcal{N}_{r_0}(\O^*)})}$, the supremum of $||(D^2_{x, y}c_0)^{-1}||$ over $\overline\O \times \overline{\mathcal{N}_{r_0}(\O^*)}$, and the constants $\delta$ and $\delta^*$ in \eqref{cconvexity} and \eqref{cstarconvexity}.
 \end{itemize}
 \end{Definition}
\begin{Definition}\label{def: c-convex function}
\begin{enumerate}[(i)]
\item A function $\varphi: \Omega\to \R$ is said to be \emph{$c$-convex} if for any  $x_0\in \Omega$, there exists $y_0\in \Omega^*$ and $z_0\in \R$ such that 
\begin{align*}
 \varphi(x_0)&=-c(x_0, y_0)+z_0,\\
 \varphi(x)&\geq -c(x, y_0)+z_0,\quad\forall x\in \Omega.
\end{align*}
\item We say $\varphi$ is \emph{strictly $c$-convex} if it is $c$-convex, and the second inequality above is strict whenever $x\neq x_0$.
\item A function $\varphi \in C^2(\overline\O)$ is said to be \emph{locally, uniformly $c$-convex} if $D^2\varphi(x) - A^c(x, \nabla \varphi(x)) > 0$ as a matrix for every $x \in \overline{\O}$, where $A^c$ is defined below.
\end{enumerate}
\end{Definition}

\begin{Remark}\label{remark: sign convention}
 We note here that the sign convention on the cost function $c$ is negative that of \cite{TrudingerWang09, Kitagawa12, AbedinKitagawa20}. We have made this choice in order to express the associated optimal transport problem as a minimization, while maintaining that our potential functions are $c$-convex (supported from below by functions constructed from $c$). However, due to our choice of sign for the matrix function $A$, all expressions in terms of $A$ will be the same as in the above mentioned references.
\end{Remark}

For a function  $u\in C^2_xC^1_t(\overline\O \times [0, \infty))$  (which, in the sequel, will be the solution to a parabolic optimal transportation problem), we will employ the following notation (if the cost function $c$ and the function $u$ are clear from context, we may suppress them from the notation at times):
\begin{enumerate}[(i)]
\item $A^c(x,p):=-D^2_x c(x, y)\vert_{y=\exp^{c}_x(p)}$
\item $T^{c, u}(x,t) = \exp^{c}_x(\nabla u(x, t))$
\item $B^c(x,p) = |\det \ (D^2_{x,y} c)(x,y)\vert_{y=\exp^{c}_x(p)}|\cdot \frac{\rho(x)}{\rho^*(\exp^{c}_x(p))}$
\item $\G^c(x,p) = \hstar(\exp^{c}_x(p))$
\item $\beta^{c, u}(x,t) = \nabla_p \G^c(x,p) \bigg|_{p = \nabla u(x, t)}$ 
\item $W^{c, u}(x,t) = D^2u(x,t) - A^c(x, \nabla u(x,t))$
\end{enumerate}
The components of the matrix $W^{c, u}(x,t)$ will be denoted by $w^{c, u}_{ij}$, while the components of the inverse matrix will be denoted by $w_{c, u}^{ij}$.

Finally, we recall the weak Ma-Trudinger-Wang condition, first introduced in \cite{MaTrudingerWang05} in a stronger form.
\begin{Definition}\label{def: MTW} The cost function $c(x,y)$ satisfies the \emph{weak Ma-Trudinger-Wang} condition if 
\begin{equation}\label{A3}\tag{weak-MTW}
D^2_{p_i p_j} A^c_{k \ell}(x,p)V^i V^j \eta^k \eta^{\ell} \geq 0 \quad \text{for all } x \in \overline{\O},\ p \in -\nabla_x c(x, \overline{\O^*}), \ V \perp \eta.
\end{equation}
\end{Definition}

\subsection{The Parabolic Optimal Transport Problem}
Using the above notation, we can now precisely state the parabolic optimal transportation problem. Given a cost function $c$, domains $\O,\O^*$ satisfying the necessary convexity conditions \eqref{cconvexity} and \eqref{cstarconvexity}, absolutely continuous measures $\rho, \rho^*$ supported on $\O,\O^*$ respectively and satisfying \eqref{massbalance} and a locally, uniformly $c$-convex (as in Definition \ref{def: c-convex function}) function $u_{\text{initial}}$ satisfying the compatibility conditions \eqref{boundarycompatbility}, we seek to find a function $u\in C^2_xC^1_t(\overline\O \times [0, \infty))$ satisfying the evolution equation
\begin{equation}\label{PDE}\tag{Par OT}
\begin{cases}
 \dot u(x,t) = \log \det (W^{c, u}(x,t)) - \log B^c(x,\nabla u(x,t)), & \quad x \in \O, \ t > 0 \\
\G^c(x,\nabla u(x,t)) = 0, & \quad x \in \partial \O, \ t > 0 \\
u(x,0) = u_{\text{initial}}(x), & \quad x \in \O.
\end{cases}
\end{equation}
Here, a dot indicates differentiation in the time variable. The function $u_{\text{initial}}$ must satisfy the compatibility conditions
\begin{equation}\label{boundarycompatbility}\tag{IC}
\begin{cases}
u_{\text{initial}} \in C^{4,\a}(\overline\O) \text{ for some } \a \in (0,1] \\
\G^c(x,\nabla u_{\text{initial}}(x)) = 0 \text{ on } \partial \Omega \\
T^c_{\text{initial}}(\O) = \O^*, \text{ where } T^c_{\text{initial}}(x) := \exp^{c}_x(\nabla u_{\text{initial}}(x)),
\end{cases}
\end{equation}
It follows from  \cite[Theorem 4.4]{Kitagawa12} that for any given $c, \O, \O^*, \rho, \rho^*$ satisfying the conditions outlined in Subsection \ref{subsec:notation}, there is a time $\tmax > 0$ depending on $c, \O, \O^*, \rho, \rho^*$, and $u_{\text{initial}}$ such that the solution $u$ of \eqref{PDE} exists on $\overline{\O} \times [0,\tmax)$. Since in this paper we will fix $\O,\O^*,\rho,\rho^*$ but vary the cost function $c$ in a neighborhood of the distinguished cost $c_0$, and construct initial data $u_{\text{initial}}$ corresponding to each choice of $c$, we will often highlight the dependence of $\tmax$ on $c$. If, in addition, the cost function $c$ satisfies the  \eqref{A3} condition, then the main result of \cite{Kitagawa12} shows that $\tmax(c) = +\infty$. As mentioned in the introduction, the main goal of this paper is to show that $\tmax(c) = +\infty$ for certain cost functions $c$ that fail to satisfy the  \eqref{A3} condition, given $\rho$ and $\rho^*$.

We end this section by collecting those uniform estimates established in  \cite{Kitagawa12}  which do not require the  \eqref{A3} condition.
\begin{Proposition}\label{prop: nonMTW estimates}
Suppose $c\in C^{4, \alpha}(\overline \O\times \mathcal{N}_{r_0}(\overline{\O^*}))$ satisfies \eqref{bi-twist}, and $\O$ and $\O^*$ are respectively uniformly $c$- and $c^*$-convex with respect to each other. Let $u\in C_x^4C_t^2(\overline \O\times [0, \tmax))$ be a solution of \eqref{PDE} for some $\tmax>0$. Then there are constants $K_1$, $K_2>0$ depending only on the structure of the problem and $\Norm{u_{\text{initial}}}_{C^2(\overline{\O})}$ such that, 
\begin{align*}
\max\left\{\sup_{(x, t)\in\overline \O\times [0, \tmax)}\norm{\nabla u(x, t)}, \sup_{(x, t)\in\overline \O\times [0, \tmax)}\norm{\dot u(x, t)} \right\} \leq K_1,\\
 \inf_{(x, t)\in\partial \O\times [0, \tmax)}\inner{\beta^{c, u}(x, t)}{\nu(x)}\geq K_2.
\end{align*} 
\end{Proposition}
\begin{proof}
 This follows immediately from \cite[Theorems 6.1, 8.1, 9.2]{Kitagawa12} noting that none of the above estimates rely on the condition \eqref{A3}. For those results which rely on \cite[Lemma 5.7]{Kitagawa12} (whose proof makes use of \eqref{A3}) we may instead use Proposition \ref{prop: injectivity} below.
\end{proof}

In \cite[Section 5]{Kitagawa12}, the condition \eqref{A3} is used to show that the map $T^c$ is injective and the inverse is related to a parabolic PDE corresponding to reversing the roles of the domains $\O$ and $\O^*$; this ``dual'' problem is used to obtain the uniform obliqueness estimates in \cite[Section 9]{Kitagawa12}. Using recent partial regularity results established in \cite{DePhilippisFigalli15}, we can circumvent the use of \eqref{A3}; for completeness, we provide an alternative proof of the relevant results of \cite[Section 5]{Kitagawa12} in the proposition below.
\begin{Proposition}\label{prop: injectivity}
 If $u$ is a solution of \eqref{PDE}, then $T^c(\cdot, t)$ is injective for each $t\in [0, \tmax(c))$. Moreover, the function $u^*: \overline{\O^*}\times [0, \tmax(c))\to \R$ defined by $u^*(y, t):=-c((T^c)^{-1}(y, t), t)-u((T^c)^{-1}(y, t), t)$ satisfies the (dual) boundary value problem
\begin{align}
\begin{cases}
 \dot u^*(y,t) = \log \det (W^{c^*, u^*}(y, t)) - \log B^{c*}(y,\nabla u^*(y,t)), & \quad x \in \O^*, \ t > 0 \\
\G^{c^*}(y,\nabla u^*(y,t)) = 0, & \quad y \in \partial \O^*, \ t > 0 \\
u^*(y,0) = -c((T^c_{\text{initial}})^{-1}(y), y)-u_{\text{initial}}((T^c_{\text{initial}})^{-1}(y)), & \quad y \in \O^*,
\end{cases}\label{eqn: dual PDE}
\end{align}
where
\begin{enumerate}[(i)]
\item $B^{c^*}(y,q) = |\det \ (D^2_{y, x} c)(x,y)_{x=\exp^{c^*}_y(q)}|\cdot \frac{\rho^*(y)}{\rho(\exp^{c^*}_y(q))}$
\item $\G^{c^*}(y,q) = h(\exp^{c^*}_y(q))$
\item $W^{c^*, u^*}(y,t) = D^2u^*(y,t) +D^2_y c(x,y)_{x=\exp^{c^*}_y(\nabla u^*(y,t))}$
\end{enumerate}
and $h$ is a normalized defining function for $\O$.
\end{Proposition}
\begin{proof}
Since $c$ is fixed, we supress the cost in the notation $T^c$ for this proof. Fix $t\in [0, \tmax(c))$. By \cite[Theorem 1.3]{DePhilippisFigalli15}, there exist closed sets $\Sigma\subset \O$ and $\Sigma^*\subset\O^*$, both of zero Lebesgue measure, such that $T(\cdot, t): \O\setminus\Sigma\to\O^*\setminus\Sigma^*$ is a homeomorphism. Now suppose that $y:=T(x_1, t)=T(x_2, t)$ for some $x_1\neq x_2\in \O$, then we must have $x_1$, $x_2\in \Sigma$ and $y\in \Sigma^*$. Since $\det D_xT(\cdot, t)\neq 0$ on $\O$, by the inverse function theorem we can see there are open sets $U_1$, $U_2\subset \O$ and $V\subset \O^*$ with $x_1\in U_1$, $x_2\in U_2$, and $y\in V$ where there exist local inverses $S_i: V\to U_i$ of $T(\cdot, t)$, $i = 1,2$; without loss we may assume $U_1\cap U_2=\emptyset$. Since $\Sigma^*$ is measure zero, there exists a point $y'\in V\setminus \Sigma^*$, and $S_i(y')\in U_i\setminus\Sigma$. However, this would imply that $T(S_1(y'), t)=y'=T(S_2(y'), t)$ which is a contradiction as $S_1(y')\neq S_2(y')\in \O\setminus \Sigma$. Thus $T(\cdot, t)$ is injective on $\O$.

By the inverse function theorem, the map $T^{-1}$ is differentiable on $\overline{\O^*}\times [0, \tmax(c))$, thus by differentiating the relation $u^*(y, t)=-c((T^c)^{-1}(y, t), t)-u((T^c)^{-1}(y, t), t)$ we can easily see $u^*$ satisfies \eqref{eqn: dual PDE}.
\end{proof}

\section{Construction of Initial Data}\label{sec : initialdata}

In this section we will work toward showing the existence of suitable initial conditions for \eqref{PDE} corresponding to a cost function $c$, when $c$ is sufficiently close to $c_0$. First we show that if $c$ is sufficiently close to $c_0$ in $C^4(\overline \O\times \overline{\mathcal{N}_{r_0}(\O^*)})$, then  $c$ and the domains involved inherit various structural properties from $c_0$.
\begin{Lemma}\label{lem: bitwist preserved}
There is a constant $R_0>0$ depending only on the structure of the problem such that:
\begin{enumerate}
 \item If $\Norm{c-c_0}_{C^4(\overline \O\times \overline{\mathcal{N}_{r_0}(\O^*)})}<R_0$, then $\O$ and $\O^*$ are uniformly $c$- and $c^*$-convex with respect to each other, with constants $\delta/2$ and $\delta^*/2$ in \eqref{cconvexity} and \eqref{cstarconvexity}.
 \item If $\Norm{c-c_0}_{C^4(\overline \O\times \overline{\mathcal{N}_{r_0}(\O^*)})}<R_0$, $c$ satisfies conditions \eqref{bi-twist} and \eqref{mixedHessianinvertible} on $\overline \O\times \overline{\mathcal{N}_{r_0}(\O^*)}$.
 \item If $u_0\in C^1(\overline \O)$ is a $c_0$-convex function such that $T^{c_0, u_0}$ is well-defined and a homeomorphism on $\overline \O$, with $\det DT^{c_0, u_0}(x)\neq 0$ for all $x\in \overline \O$, and $\Norm{u-u_0}_{C^{1}(\overline \O)}<R_0$, then $T^{c, u}$ is well-defined, and a homeomorphism on $\overline \O$; here $R_0$ may also depend on $u_0$.
\end{enumerate}
\end{Lemma}
\begin{proof}
Claim (1) above and the fact that $c$ will satisfy \eqref{mixedHessianinvertible} are clear if $r_0$ is sufficiently small, depending only on the structure of the problem.

 Suppose by contradiction that some cost fails \eqref{bi-twist} for $R_0$ arbitrarily small. Then there exists a sequence of cost functions $c_k$ converging to $c_0$ in $C^4(\overline \O\times \overline{\mathcal{N}_{r_0}(\O^*)})$, and sequences of points $x_k\in \overline \O$, $y_{1, k}\neq y_{2, k}\in \overline{\mathcal{N}_{r_0}(\O^*)}$ such that $-\nabla_x c_k(x_k, y_{1, k})=-\nabla_x c_k(x_k, y_{2, k})$. By compactness, we may pass to subsequences and assume that $x_k$, $y_{1, k}$, $y_{2, k}$ converge respectively to points $x_\infty$, $y_{1, \infty}$, and $y_{2, \infty}$. This implies $-\nabla_x c_0(x_\infty, y_{1, \infty})=-\nabla_x c_0(x_\infty, y_{2, \infty})$, and since $c_0$ satisfies \eqref{bi-twist}, we must have $y_{1, \infty}=y_{2, \infty}=:y_\infty$. Since for $k$ large we may assume $c_k$ satisfies \eqref{mixedHessianinvertible}, we see $D^2_{x, y}c_k(x_k, y)$ is invertible for any $y\in \overline{\mathcal{N}_{r_0}(\O^*)}$, moreover we can see that the operator norms of these inverses and the Lipschitz constants of the mappings $y\mapsto D^2_{x, y}c_k(x_k, y)$ are bounded by constants depending only the structure of the problem. Since $\O^*$ is $c^*$-convex by the first claim in the lemma, it has Lipschitz boundary; we may thus combine \cite[Theorem 1]{Dederick13} with \cite[Chapter XIV, Lemma 1.3]{Lang93} to see there is neighborhood around $y_\infty$ on which all of the mappings $y\mapsto -\nabla_xc_k(x_k, y)$ are invertible. This is a contradiction for $k$ large enough that both $y_{1, k}$ and $y_{2, k}$ belong to this neighborhood, hence we obtain claim (2). Claim (3) follows in a similar manner, allowing for the extra dependency on $u_0$.
\end{proof}

\begin{Lemma}\label{lem: neighborhood twist}
Suppose $c\in C^4(\overline\O\times \overline{\mathcal{N}_{r_0}(\O^*)})$ with $\Norm{c-c_0}_{C^4(\overline\O\times \overline{\mathcal{N}_{r_0}(\O^*)})}<\min\{R_0, \frac{r_0}{6\sup \Norm{(D^2_{x, y}c_0)^{-1}}}\}$ where $R_0$ is the constant from Lemma \ref{lem: bitwist preserved} (2). Then 
$$\mathcal{N}_{\frac{r_0}{2\sup \Norm{(D^2_{x, y}c_0)^{-1}}}}(-\nabla_xc_0(x, \overline{\O^*}))\subset -\nabla_xc(x, \overline{\mathcal{N}_{r_0}(\O^*)}) \quad \forall \ x\in \overline \O.$$
\end{Lemma}
\begin{proof}
  Fix $x\in \overline \O$, and let us write 
\begin{align*}
 S_0(y):&=-\nabla_xc_0(x, y),\quad
 S(y):=-\nabla_xc(x, y).
\end{align*}
By Lemma \ref{lem: bitwist preserved} (2), $S$ and $S_0$ are both homeomorphisms on $\overline{\mathcal{N}_{r_0}(\O^*)}$.

Suppose that $q\in \mathcal{N}_{\frac{r_0}{2\sup \Norm{(D^2_{x, y}c_0)^{-1}}}}(S_0(\overline{\O^*}))\setminus S(\overline{\mathcal{N}_{r_0}(\O^*)})$. Then there exists $y\in \overline{\O^*}$ such that $\norm{q-S_0(y)}<\frac{r_0}{2\sup \Norm{(D^2_{x, y}c_0)^{-1}}}$. Also since $\mathcal{N}_{r_0}(\O^*)$ has  $C^1$ boundary and $S$ is a homeomorphism, there is an $s\in [0, 1)$ such that $q_s:=(1-s)S(y)+s q\in \partial S(\mathcal{N}_{r_0}(\O^*))=S(\partial \mathcal{N}_{r_0}(\O^*))$. In particular, there exists $y_s\in \partial  \mathcal{N}_{r_0}(\O^*)$ such that  $q_s=S(y_s)$. Then, 
\begin{align*} 
& \frac{r_0}{2\sup \Norm{(D^2_{x, y}c_0)^{-1}}}+\Norm{S_0-S}_{C^0(\overline\O)}\geq\norm{q-S(y)}> \norm{q_s-S(y)} \\
& =\norm{S(y_s)-S(y)}\\
&\geq \norm{S_0(y)-S_0(y_s)}-\norm{S_0(y_s)-S(y_s)}-\norm{S_0(y)-S(y)}\\
&\geq [S_0^{-1}]^{-1}_{C^{0, 1}(\overline{S_0(\mathcal{N}_{r_0}(\O^*))})}\norm{y-y_s}-2\Norm{S_0-S}_{C^0(\overline\O)}\\
&\geq  [S_0^{-1}]^{-1}_{C^{0, 1}(\overline{S_0(\mathcal{N}_{r_0}(\O^*))})}r_0-2\Norm{S_0-S}_{C^0(\overline\O)},
\end{align*}
but this is a contradiction after rearranging since $[S_0^{-1}]_{C^{0, 1}(\overline{S_0(\mathcal{N}_{r_0}(\O^*))})}\leq \sup \Norm{(D^2_{x, y}c_0)^{-1}}$ and $\Norm{S_0-S}_{C^0(\overline\O)}\leq \Norm{c-c_0}_{C^4(\overline\O\times \overline{\mathcal{N}_{r_0}(\O^*)})}$.
\end{proof}
Finally, we use an implicit function theorem argument to show suitable initial conditions for a cost $c$ exist given a ``good'' initial condition for $c_0$.
\begin{Lemma}\label{lem : initial}
Suppose $u_0\in C^{2, \alpha}(\overline\O)$ for some $\alpha\in (0, 1]$, is strictly $c_0$-convex, is locally, uniformly $c_0$-convex,  and satisfies $\int_{\O} u_0(x) \ dx = 0$. Assume, in addition, that $T^{c_0, u_0}$ is a homeomorphism between $\overline \O$ and $\overline {\O^*}$, and $\det DT^{c_0, u_0}\neq 0$ on $\overline \O$. Then there exists $R_1>0$ depending on the structure of the problem and $u_0$, and a continuous mapping $\Psi: B^{C^4(\overline \O\times \overline{\mathcal{N}_{r_0}(\O^*)})}_{R_1}(c_0)\to C^{2, \alpha}(\overline \O)$ such that for any $c\in B^{C^4(\overline \O\times \overline{\mathcal{N}_{r_0}(\O^*)})}_{R_1}(c_0)$,
\begin{align}
&\Psi(c)\text{ is strictly }c\text{-convex and locally, uniformly }c\text{-convex}.\label{eqn: psi(c) convex}\\
&T^{c, \Psi(c)}\text{ is a homeomorphism between }\overline \O\text{ and }\overline {\O^*}.\label{eqn: psi(c) homeo}
\end{align}

\end{Lemma}
\begin{proof}
For ease of notation, during this proof we will supress the first variable in the notation for $G^c$. Define $\mathcal{B}_1:=\overline{B^{C^4(\overline \O\times \overline{\mathcal{N}_r(\O^*)})}_{\tilde{r}_0}(c_0)}$, where $\tilde{r}_0<\min\{R_0, \frac{r_0}{6\sup \Norm{(D^2_{x, y}c_0)^{-1}}}\}$ and $R_0$ is the radius obtained from (1)-(3) of Lemma \ref{lem: bitwist preserved}, hence depends on the structure of the problem and $u_0$.

Now let $\mathcal{B}_2:=\{u\in \overline{B^{C^{2, \alpha}(\overline \O)}_{\tilde{r}_0}(u_0)}\mid \int_\Omega u=0\}$, and define the map $\Phi: \mathcal{B}_1\times \mathcal{B}_2\to C^{0, \alpha}(\overline \O)\times C^{1, \alpha}(\partial \O)$ by
\begin{align*}
 \Phi(c, u):=(\Delta (u-u_0), G^c(\nabla u)\vert_{\partial \Omega}).
\end{align*}
Note that $\Phi(c_0,u_0) = (0,0)$. Since $T^{c_0, u_0}$ is a homeomorphism between $\overline\O$ and $\overline{\O^*}$, by Lemma \ref{lem: neighborhood twist} we have for any $x\in \overline \O$ and $u\in \mathcal{B}_2$,
\begin{align*}
 \nabla u(x)\in \mathcal{N}_{\tilde{r}_0}(\nabla u_0(x))&=\mathcal{N}_{\tilde{r}_0}(-\nabla_xc_0(x, T^{c_0, u_0}(x)))\subset -\nabla_xc(x, \overline{\mathcal{N}_{r_0}(\O^*)}),
\end{align*}
in particular $G^c(x, \nabla u(x))$ is well-defined for all $x\in \overline \O$. 
Denoting the Fr\'echet derivative of the map $u\mapsto \Phi(c_0, u)$ at $u_0$ by $D_u\Phi(c_0, u_0)$, we claim that $D_u\Phi(c_0, u_0)$ is injective; here we consider $\mathcal{B}_2$ as a closed subset of the Banach space $\mathcal{B}:=\{u\in C^{2, \alpha}(\overline \O)\mid \int_\Omega u=0\}$ with the $C^{2, \alpha}$ norm, then $u_0$ is an interior point of $\mathcal{B}_2$. First note that an explicit calculation of the Fr\'echet derivative of $\Phi$ (this requires $C^4$ smoothness of the cost $c_0$) shows that for any $\phi\in \mathcal{B}$, 
\begin{align}\label{eqn: frechet derivative}
 D_u\Phi(c_0, u_0)(\phi)=(\Delta \phi, \inner{D_pG^{c_0}(\nabla u_0)}{\nabla \phi}\vert_{\partial \Omega}).
\end{align}
By Proposition \ref{prop: nonMTW estimates} $G^{c_0}(\cdot, \nabla u_0)$ is uniformly oblique on $\partial \O$, hence by Hopf's lemma combined with standard existence theory for Poisson's equation (for example, \cite[Theorem 6.31]{GilbargTrudinger01})  we see $D_u\Phi(c_0, u_0)$ is a bijection between $\mathcal{B}$ and   $C^{0, \alpha}(\overline \O)\times C^{1, \alpha}(\partial \O)$. Thus by the implicit mapping theorem \cite[Theorem 4.B]{Zeidler86v1}, there is a bijective, continuous mapping $\Psi: B^{C^4(\overline \O\times \overline{\mathcal{N}_r(\O^*)})}_{R_1}(c_0)\to \mathcal{V}$ for some radius $R_1 > 0$ and neighborhood $\mathcal{V}$  of $u_0$ in $\mathcal{B}$ such that $\Psi(c_0) = u_0$ and $\Phi(c, \Psi(c))=(0, 0)$ for all $c \in B^{C^4(\overline \O\times \overline{\mathcal{N}_r(\O^*)})}_{R_1}(c_0)$. We may further shrink $R_1$, but it will always be in a manner that depends only on the structure of the problem and $u_0$, hence we can also assume that the size of the neighborhood $\mathcal{V}$ also depends only on these parameters. We remark here this implies that the quantities $\Norm{c}_{C^4(\overline \O\times \overline{\mathcal{N}_{r_0}(\O^*)}}$, $\Norm{u}_{C^{2, \alpha}(\overline\O)}$, and $\Norm{(D^2_{x, y}c)^{-1}}_{C^4(\overline \O\times \overline{\mathcal{N}_{r_0}(\O^*)}}$ have upper bounds that also depend only on the structure of the problem and $u_0$, provided $R_1$ is taken small enough.

Now, for any  $c\in B^{C^4(\overline \O\times \overline{\mathcal{N}_{r_0}(\O^*)})}_{R_1}(c_0)$, we have $G^c(\cdot, \nabla \Psi(c))\equiv 0 \text{ on }\partial \Omega$; this implies  $T^{c, \Psi(c)}(\partial \O)\subset \partial \O^*$. Possibly taking $R_1$ even smaller, by Lemma \ref{lem: bitwist preserved} (3) we may assume that $T^{c, \Psi(c)}$ is a homeomorphism on $\overline \O$, hence a homeomorphism on $\partial \O$. Since $\O$ is uniformly $c^*$-convex with respect to $\O^*$ by Lemma \ref{lem: bitwist preserved} (1), $\O$ is homeomorphic to a ball, in particular $\partial \O$ is connected thus we must have $T^{c, \Psi(c)}(\partial \O)= \partial \O^*$. Now suppose for some $x\in \O$ we have $T^{c, \Psi(c)}(x)\not \in \O^*$. Then writing $c_s:=(1-s)c_0+s c$, by continuity of the map $\Psi$ the curve $\{T^{c_s, \Psi(c_s)}(x)\mid s\in [0, 1]\}$ is continuous with $T^{c_0, \Psi(c_0)}(x)
\in \O^*$. Thus for some value of $s$, we will have $T^{c_s, \Psi(c_s)}(x)\in \partial \O^*$ contradicting that $T^{c_s, \Psi(c_s)}$ is a homeomorphism of $\overline \O$ with $\overline \O^*$, and of $\partial \O$ with $\partial \O^*$. Thus $T^{c, \Psi(c)}(\O)=\O^*$ proving the claim \eqref{eqn: psi(c) homeo}.

We will now show \eqref{eqn: psi(c) convex}. Fix $c\in B^{C^4(\overline \O\times \overline{\mathcal{N}_{r_0}(\O^*)})}_{R_1}(c_0)$ and write $u:=\Psi(c)$. Since $u_0$ is locally, uniformly $c_0$-convex, by continuity and compactness the smallest eigenvalues of the matrices $D^2u_0(x)-A^{c_0}(x, \nabla u_0(x))$ have a strictly positive lower bound independent of $x\in\overline\O$. Then there exists some $R>0$ depending only on the structure of the problem and $u_0$ such that $D^2u_0(x)-A^{c_0}(x, \nabla u_0(x_0))$ also has a strictly positive bound independent of $x$ and $x_0$, whenever $\norm{x-x_0}<R$. Next, possibly shrinking $R_1$, we can assume there exists some $\epsilon_0>0$ depending only on the structure of the problem and $u_0$ such that
\begin{align}
 D^2u(x)-A^c(x, \nabla u(x_0))\geq \epsilon_0\Id,\quad \forall \ x_0, x\in\overline \O,\ \norm{x-x_0}<R.\label{eqn: loc unif c conv preserved}
\end{align}
In particular, $u$ is locally, uniformly $c$-convex. Now fix $x$, $x_0 \in \overline\O$ and let $y_0 := \exp^c_{x_0}(\nabla u(x_0))$ which belongs to $\overline{\O^*}$ by above. First suppose $x\neq x_0\in \overline \O$ is such that $\norm{-\nabla_xc(x, y_0)+\nabla_xc(x_0, y_0)}<\frac{R}{M_1}$ for an $M_1$ depending only the structure of the problem such that $\sup\Norm{(D^2_{x, y}c)^{-1}}<M_1$. Then write for $s\in [0, 1]$,
\begin{align*}
p_0 &:=-\nabla_y c(x_0, y_0)=\nabla u(x_0), \\
p & :=-\nabla_y c(x, y_0),\\
x(s)& :=\exp_{y_0}^{c^*}((1-s)p_0+s p).
\end{align*}
Note by Lemma \ref{lem: bitwist preserved} (1), the set $-\nabla_yc(\overline \O, y_0)$ is uniformly convex, hence $x(s)\in \overline\O$. A quick calculation yields
\begin{align*}
 \dot x_i(s)&=-c^{k, i}(p-p_0)_k,\\
 \ddot x_i(s)&=-c^{j, i}c_{kl, j}c^{m, k}c^{n, l}(p-p_0)_m(p-p_0)_n
\end{align*}
where all terms involving $c$ are evaluated at $(x(s), y_0)$; here subscripts before and after a comma indicate partial derivatives with respect to $x$ and $y$ respectively, and $c^{i, j}$ is the $(i, j)$-entry of the inverse of the matrix $D^2_{x, y}c$. Using this we find that for any $s\in [0, 1]$,
\begin{align*}
 \norm{x(s)-x_0}&\leq \int_0^s \norm{\dot x(s)}ds \leq \sup\Norm{(D^2_{x, y}c)^{-1}}\norm{p-p_0}<\frac{R\sup\Norm{(D^2_{x, y}c)^{-1}}}{M_1}<R.
\end{align*}
Using Taylor expansion, we thus find some $\tilde s\in [0, 1]$ such that
\begin{align}
 &u(x)-(-c(x, y_0)+c(x_0, y_0)+u(x_0)) \notag \\
= \ &\inner{\nabla u(x_0)+\nabla_xc(x_0, y_0)}{\dot x(0)}\notag\\
 & +\inner{[D^2u(x(\tilde{s}))+D^2_xc(x(\tilde{s}), y_0)]\dot x(\tilde{s})}{\dot x(\tilde{s})}+\inner{\nabla u(x(\tilde{s}))+\nabla_xc(x(\tilde{s}), y_0)}{\ddot x(\tilde{s})}\notag\\
\geq \ & \epsilon_0\norm{((-D^2_{x, y} c(x_0, y_0))^T)^{-1}(p-p_0)}^2\notag\\
 &-\sup\Norm{((D^2_{x, y}c)^T)^{-1}}\Norm{c}_{C^4(\overline \O\times \overline {\mathcal{N}_r(\O^*)})}\norm{\nabla u(x(\tilde{s}))+\nabla_xc(x(\tilde{s}), y_0)}\norm{((-D^2_{x, y} c(x(\tilde{s}), y_0))^T)^{-1}(p-p_0)}^2 \notag \\
& \label{eqn: strict c-conv1}
\end{align}
where we have used that $\nabla u(x_0)+\nabla_xc(x_0, y_0)=0$. We can also calculate
\begin{align*}
 \norm{\nabla u(x(\tilde{s}))+\nabla_xc(x(\tilde{s}), y_0)}&=\norm{\nabla u(x(\tilde{s}))-[(1-\tilde{s})p_0+\tilde{s} p]} \\
& \leq \norm{\nabla u(x(\tilde{s}))-\nabla u(x_0)}+\tilde{s}\norm{p-p_0}\\
&\leq \sup\Norm{D^2u}\sup\norm{\dot x}+\tilde{s}\norm{p-p_0}\\
& \leq (1+ \sup\Norm{D^2u}\sup\Norm{((D^2_{x, y}c)^T)^{-1}})\norm{p-p_0}\\
 &\leq \frac{R(1+ \sup\Norm{D^2u}\sup\Norm{((D^2_{x, y}c)^T)^{-1}})}{M_1}.
\end{align*}
Thus by taking $M_1$ large enough and combining with \eqref{eqn: strict c-conv1}, we obtain for some $\epsilon_1>0$ depending only on the structure of the problem and $u_0$ (note here that $\sup\Norm{D^2u}$ is controlled by the $C^2$ norm of $u_0$ and the size of the neighborhood $\mathcal{V}$)
\begin{align}
 u(x)-(-c(x, y_0)+c(x_0, y_0)+u(x_0))&\geq \epsilon_1\norm{p-p_0}^2>0.
\end{align}
Now suppose $\norm{-\nabla_xc(x, y_0)+\nabla_xc(x_0, y_0)}\geq\frac{R}{M_1}$. Since $u_0$ is strictly $c_0$-convex, the function
\begin{align*}
 \underline u(z, z_0):&=u_0(z)-(-c_0(z, T^{c_0, u_0}(z_0))+c_0(z_0, T^{c_0, u_0}(z_0))+u_0(z_0))
\end{align*}
is strictly positive on the set $\mathcal{D}_0:=\{(z, z_0)\in \overline \O\times \overline \O\mid \norm{-\nabla_xc_0(z, T^{c_0, u_0}(z_0))+\nabla_xc_0(z_0, T^{c_0, u_0}(z_0))}\geq\frac{R}{2M_1}\}$, which is seen to be compact by the continuity of $T^{c_0, u_0}$ and boundedness of $\O$. In particular, $\inf_{\mathcal{D}_0}\underline u>0$. Possibly shrinking $R_1$, we can ensure that $(x, x_0)\in \mathcal{D}_0$, hence
\begin{align*}
 u(x)-(-c(x, y_0)+c(x_0, y_0)+u(x_0))&\geq \underline{u}(x, x_0)-2(\Norm{u-u_0}_{C^0(\overline \O)}+\Norm{c-c_0}_{C^0(\overline\O)})>0,
\end{align*}
possibly shrinking $R_1$ further. This shows that $\Psi(c)$ is strictly $c$-convex, finishing the proof of \eqref{eqn: psi(c) convex}.
\end{proof}

\section{$C^2$ Estimates}\label{sec : c2estimate}
Our goal in this section is to establish global $C^2$ estimates for solutions of \eqref{PDE} independent of the time of existence $\tmax(c)$. Such estimates are essential for obtaining the infinite-time existence of the flow. In previous work of the second author, \cite{Kitagawa12}, such uniform estimates were obtained by making crucial use of the assumption that the cost function satisfies the  \eqref{A3} condition. However, since this condition is not preserved under $C^4$ perturbations of the cost function, we must revisit the proof in \cite[Section 10]{Kitagawa12}. Our strategy is to establish a dichotomy for the operator norm of the matrix $W^{c,u}$. This argument is similar in spirit to one carried out in the elliptic case by Warren in \cite{Warren11}. However, we note that the approach of \cite{Warren11} makes heavy use of the log-concavity of the source and target measures in various barrier constructions and maximum-principle type arguments. Additionally, we mention that since our condition \eqref{MTWbound} below is a perturbation of \eqref{A3}, we must necessarily take an  approach \`a la Pogorelov, using an auxiliary function similar to the one used in \cite{TrudingerWang09}. This is in sharp contrast to the elliptic $C^2$ estimates in \cite{MaTrudingerWang05} and the parabolic $C^2$ estimates in \cite{KimStreetsWarren12}, both of which rely on a stronger version of the MTW condition. In particular, we note that we are not able to prove the necessary estimates simply by choosing $w_{ij}\xi^i\xi^j$ as the auxiliary function in the proof of Proposition \ref{prop : dichotomy} below. As our argument will illustrate, additional lower order terms need to be introduced in the auxiliary function to establish the desired polynomial inequality for the quantity $\sup ||W^{c,u}||$.

Throughout this section, the cost function $c$ will remain fixed, and $u$ will be the solution of \eqref{PDE} corresponding to this fixed cost function $c$ that exists up to $\tmax(c)$; thus we omit the dependence on $c$ from various pieces of notation. We will make use of the following linearization of \eqref{PDE}:
\begin{equation}\label{linear}
\mathcal{L}\theta := w^{ij}\left(\theta_{ij} - D_{p_k} A_{ij} \theta_k \right) - D_{p_k} (\log B) \theta_k -  \dot\theta 
\end{equation}
where in the coefficients, $p = \nabla u(x,t)$; note that differentiating \eqref{PDE} in time shows  $\mathcal{L}\dot u=0$. We will also assume the key condition which replaces the \eqref{A3} condition, namely that for some $\MTWconst>0$, 
\begin{align}\label{MTWbound}
 D^2_{p_kp_\ell}A_{ij}V^iV^j \eta^k\eta^l\geq -\MTWconst \norm{V}^2\norm{\eta}^2,\quad \forall \ V\perp \eta.
\end{align}
Note that \eqref{MTWbound} makes no assumption on the quantity appearing in the left-hand side when $V, \eta$ are not orthogonal.

We begin with a barrier construction for the linear operator \eqref{linear}. This is similar to the construction of an elliptic barrier in \cite[Lemma 2.2]{JiangTrudinger14}, but instead of using the condition \eqref{A3}, we will utilize the bound \eqref{MTWbound}.
\begin{Lemma}\label{lem: barrier}
 There exists a constant $K_0>0$ depending on the structure of the problem so that, as long as $\MTWconst<K_0$, there exists a function $\psi\in C^2_xC^1_t(\overline \Omega\times [0, \tmax))$ such that for all $(x, t)\in \overline \O\times [0, \tmax)$,
\begin{align*}
 \mathcal{L}\psi(x, t)&\geq C_1\tr (W^{-1}(x, t))-C_2,\\
 0&<C_3^{-1} \leq\psi(x, t)\leq C_3,
\end{align*}
for some constants $C_1, C_2, C_3 >0$ depending on the structure of the problem, and $\Norm{u_{\text{initial}}}_{C^2(\overline{\O})}$ but independent of $\tmax$.
\end{Lemma}
\begin{proof}
 Let $\bar u(x)$ be the function constructed in \cite[Lemma 2.1]{JiangTrudinger14} with the choice $g(x, y, z)=-c(x, y)-z$; note that the proof there does not require the  \eqref{A3} condition. Also, let 
\begin{align*}
 U(t):= \frac{1}{|\O|} \int_\Omega u(x, t) \ dx,
\end{align*}
which belongs to $C^1([0, \tmax))$ by Proposition \ref{prop: nonMTW estimates}. For any fixed $x_0\in \Omega$, define the function $\bar u_\epsilon(x):=\bar u(x)-\frac{\epsilon}{2}\norm{x-x_0}^2$. Reasoning similarly to \cite[Lemma 2.2]{JiangTrudinger14}, we have for any sufficiently small $\epsilon>0$ a constant $a_1\in \R$ independent of $\tmax$ such that
\begin{align}\label{eqn: ubar bound}
 \log(\det D^2\bar u_\epsilon-A(x, \nabla \bar u_\epsilon))\geq a_1,
\end{align}
and that $\nabla \bar u_\epsilon(x)\in -\nabla_xc(x, \Omega^*)$ for any $x\in \Omega$. Then, for some $K>0$ to be determined, we define
\begin{align*}
 \psi(x, t):=e^{K(U(t)-u(x, t)+\bar u(x))}.
\end{align*}
We claim $\psi$ satisfies uniform bounds above and below independent of $\tmax$. To show this, it suffices to establish a bound on the term $U(t)-u(x_0,t)$ for any $x_0 \in \O$ and $t \in [0,\tmax)$. So we fix such an $x_0$ and $t$, and let $y_0 := \exp^c_{x_0}(\nabla u(x_0,t))$. Let $x \in \O$ be another arbitrary point. Then for $s\in [0, 1]$, we let
\begin{align*}
p_0 &:=-\nabla_y c(x_0, y_0)=\nabla u(x_0,t), \\
p & :=-\nabla_y c(x, y_0),\\
x(s)& :=\exp_{y_0}^{c^*}((1-s)p_0+s p).
\end{align*}
Since $\dot x_i(s)=-c^{k, i}\bigg|_{(x(s), y_0)}(p-p_0)_k$, we have 
\begin{align*}
u(x,t) - u(x_0,t) & = \int_0^1 \frac{d}{ds} u(x(s),t) \ ds \\
& = \int_0^1 \left\langle \nabla u(x(s),t), \dot x (s) \right\rangle \ ds\\
& \leq \sup_{\overline \O\times [0, \tmax)}\norm{\nabla u}\sup_{\overline \O\times \overline \O^*}\Norm{(D^2_{x, y}c)^{-1}}\norm{p-p_0} \\
& \leq \sup_{\overline \O\times [0, \tmax)}\norm{\nabla u}\sup_{\overline \O\times \overline \O^*}\Norm{(D^2_{x, y}c)^{-1}}2\sup_{\overline \O\times \overline \O^*}\norm{\nabla_{y} c}.
\end{align*}
Consequently
\begin{align*}
|U(t) - u(x_0,t)|&  \leq \frac{1}{|\O|} \int_{\O} |u(x,t) - u(x_0,t)| \ dx\\
& \leq \frac{2}{\norm{\O}}\sup_{\overline \O\times [0, \tmax)}\norm{\nabla u}\sup_{\overline \O\times \overline \O^*}\Norm{(D^2_{x, y}c)^{-1}}\sup_{\overline \O\times \overline \O^*}\norm{\nabla_{y} c} \quad \text{for any } t \in [0,\tmax).
\end{align*}
Since $\bar{u}$ is uniformly bounded on $\O$, using Proposition \ref{prop: nonMTW estimates} there exists a constant $\Lambda > 0$ depending only on the structure of the problem such that
\begin{equation}\label{boundsonpsi}
e^{-K \Lambda} \leq \psi(x,t) \leq e^{K \Lambda} \quad \text{for all } (x,t) \in \overline{\O} \times [0,\tmax).
\end{equation}
Recall the linearized operator $\mathcal{L}$ defined in \eqref{linear}. Denoting 
$$\tilde{\mathcal{L}}v:=w^{ij}\left(v_{ij} - D_{p_k} A_{ij}(\cdot, \nabla u) v_k \right),$$
we have for a constant $C>0$ depending only on the structure of the problem and $\Norm{u_{\text{initial}}}_{C^2(\overline \O)}$,
\begin{align}
 \mathcal{L}\psi&=\tilde{\mathcal{L}}\psi-K\psi(\dot U-\dot u+D_{p_k}B(x, \nabla u)(\bar u_k-u_k))\notag\\
 &\geq \tilde{\mathcal{L}}\psi-CKe^{K\Lambda},\label{eqn: barrier1}
\end{align}
where we have used Proposition \ref{prop: nonMTW estimates}. Writing $\eta(x, t)=U(t)-u(x, t)+\bar u(x)$ for ease of notation, we calculate
\begin{equation}\label{eqn: barrier5}
 \tilde{\mathcal{L}}\psi=K\psi(Kw^{ij}\eta_i\eta_j+\tilde{\mathcal{L}}\eta).
\end{equation}
Then, 
\begin{align}
 \tilde{\mathcal{L}}\eta = \ &\tilde{\mathcal{L}}\left(\frac{\epsilon}{2}\norm{x-x_0}^2\right)+\tilde{\mathcal{L}}\bar u_\epsilon+\tilde{\mathcal{L}}(U(t)-u(x, t))\notag\\
 = \ &\epsilon \tr W^{-1}-\epsilon w^{ij}D_{p_k}A_{ij}(x, \nabla u)(x-x_0)_k\notag\\
 &+w^{ij}((\bar u_\epsilon)_{ij}-D_{p_k}A_{ij}(x, \nabla u)(\bar u_\epsilon)_k)-w^{ij}(u_{ij}-D_{p_k}A_{ij}(x, \nabla u)u_k)\notag\\
 = \ &\epsilon \tr W^{-1}-\epsilon w^{ij}D_{p_k}A_{ij}(x, \nabla u)(x-x_0)_k\notag\\
 &+w^{ij}([(\bar u_\epsilon)_{ij}-A_{ij}(x, \nabla \bar u_\epsilon)]-[u_{ij}-A_{ij}(x, \nabla u)])\notag\\
 &+w^{ij}(A_{ij}(x, \nabla \bar u_\epsilon)-A_{ij}(x, \nabla u)-D_{p_k}A_{ij}(x, \nabla u)((\bar u_\epsilon)_k-u_k).\label{eqn: barrier2}
\end{align}
Let us write $\circled{I}:=-\epsilon w^{ij}D_{p_k}A_{ij}(x, \nabla u)(x-x_0)_k$, and $\circled{II}$ and $\circled{III}$ respectively for the final two lines above. By concavity of $\log\det$, we obtain using \eqref{eqn: ubar bound}  and \eqref{PDE},
\begin{align}
 \circled{II}&\geq \log\det(D^2\bar u_\epsilon-A(x, \nabla \bar u_\epsilon))-\log\det(D^2 u-A(x, \nabla  u))\notag\\
 &\geq a_1-\dot u(x, t)-
 \log B(x, \nabla u)\geq -C\label{eqn: barrier3}
\end{align}
for some constant by Proposition \ref{prop: nonMTW estimates}. By the $c^*$-convexity of $\O^*$ with respect to $\O$, the point $q_s:=(1-s)\nabla \bar u_\epsilon+s \nabla u$ belongs to $-\nabla_xc(x, \O^*)$ for any $s\in [0, 1]$. Then by Taylor expanding, for some $s\in [0, 1]$ we obtain
\begin{align*}
 \circled{III}&=\frac{w^{ij}}{2}D^2_{p_kp_\ell}A_{ij}(x, q_s)((\bar u_\epsilon)_k-u_k)((\bar u_\epsilon)_\ell-u_\ell).
\end{align*}
By choosing $x=x_0$, we then obtain
\begin{align*}
 \circled{III}&=\frac{w^{ij}}{2}D^2_{p_kp_\ell}A_{ij}(x_0, q_s)(\bar u_k-u_k)(\bar u_\ell-u_\ell),
\end{align*}
note for this choice, we also have $\circled{I}=0$.

Now for ease of notation, let us write
\begin{align*}
\MTW(V, W, \eta, \xi):=D^2_{p_kp_\ell} A_{ij}(x_0, q_s)V^iW^j\eta^k\xi^\ell,
\end{align*}
and after fixing $t\in [0, \tmax)$ make a rotation of coordinates to diagonalize the matrix $W(x_0, t)$. Letting $\{e_i\}_{i=1}^n$ be the standard basis in $\Rn$ and using the above calculations, we thus  obtain
\begin{align}
 \tilde{\mathcal{L}}\eta(x_0, t)\geq & \ \epsilon \tr (W^{-1}(x_0, t))-C\notag\\
 &+\frac{1}{2}\sum_{i=1}^nw^{ii}\MTW(e_i, e_i, \nabla (\bar u-u), \nabla(\bar u-u)).\label{eqn: barrier4}
\end{align}
Denoting $\nabla^\perp_i(\bar u-u):=\nabla(\bar u-u)-(\bar u_i-u_i)e_i$, and using \eqref{MTWbound} we calculate
\begin{align*}
&\sum_{i=1}^nw^{ii}\MTW(e_i, e_i, \nabla (\bar u-u), \nabla(\bar u-u)) \\
= & \sum_{i=1}^nw^{ii}\left[\MTW(e_i, e_i, \nabla^\perp_i (\bar u-u), \nabla^\perp_i(\bar u-u))\right.\\
& \ \left.+ \ 2(\bar u-u)_i\MTW(e_i, e_i, e_i, \nabla^\perp_i(\bar u-u))+(\bar u_i-u_i)^2\MTW(e_i, e_i, e_i,e_i)\right]\\
\geq & \sum_{i=1}^nw^{ii}\left[-\MTWconst \norm{\nabla^\perp_i (\bar u-u)}^2+(\bar u_i-u_i)^2\MTW(e_i, e_i, e_i,e_i)\right.\\
& \ \left.+ \ 2\sum_{j\neq i}(\bar u_i-u_i)(\bar u_j-u_j)\MTW(e_i, e_i, e_i, e_j)\right]\\
\geq & \sum_{i=1}^nw^{ii}\left[-\MTWconst \Norm{\nabla(\bar u-u)}^2-(\bar u_i-u_i)^2\Norm{D^2_pA}-2\Norm{D^2_pA}\sum_{j\neq i}\left(\tilde\epsilon(\bar u_j-u_j)^2+\frac{1}{4\tilde\epsilon}(\bar u_i-u_i)^2\right)\right]\\
\geq & -(\MTWconst+2\tilde\epsilon\Norm{D^2_pA})\Norm{\nabla(\bar u-u)}^2\tr (W^{-1})-\Norm{D^2_pA}(1+\frac{n-1}{2\tilde\epsilon})\sum_{i=1}^nw^{ii}(\bar u_i-u_i)^2
\end{align*} 
for $\tilde{\e} > 0$ to be determined. It follows from \eqref{eqn: barrier4} that 
\begin{align*}
\tilde{\mathcal{L}}\eta(x_0, t)\geq &\ \tr (W^{-1}(x_0, t))\left[\epsilon -(\MTWconst + 2\tilde\epsilon\Norm{D^2_pA})\Norm{\nabla(\bar u-u)}^2 \right]\\
 &-\Norm{D^2_pA}\left(1+\frac{n-1}{2\tilde\epsilon}\right)\sum_{i=1}^nw^{ii}(\bar u_i-u_i)^2.
\end{align*}
We choose $\MTWconst$ and $\tilde \epsilon$ small enough so that 
$$\epsilon -(\MTWconst + 2\tilde\epsilon\Norm{D^2_pA})\Norm{\nabla(\bar u-u)}^2\geq \epsilon/2.$$
Therefore, by \eqref{eqn: barrier5}, we have
$$\tilde{\mathcal{L}}\psi(x_0,t) \geq K\psi(x_0, t)\left(\sum_{i=1}^nw^{ii}\left[K-\Norm{D^2_pA}\left(1+\frac{n-1}{2\tilde\epsilon}\right)\right](\bar u_i-u_i)^2 + \frac{\e}{2} \tr(W^{-1}(x_0,t))\right)$$
We then choose $K$ large enough so that 
$$K-\Norm{D^2_pA}\left(1+\frac{n-1}{2\tilde\epsilon}\right)\geq 0.$$
Consequently,
$$\tilde{\mathcal{L}}\psi(x_0,t)\geq K\psi(x_0, t)\frac{\e}{2} \tr(W^{-1}(x_0,t)).$$
Finally, using \eqref{eqn: barrier1}, and then \eqref{boundsonpsi}, we conclude that
\begin{align*}
\mathcal{L}\psi(x_0,t)&\geq K\psi(x_0, t)\frac{\e}{2} \tr(W^{-1}(x_0,t)) - CKe^{K\Lambda}\\
&\geq Ke^{-K\Lambda}\frac{\e}{2} \tr(W^{-1}(x_0,t)) - CKe^{K\Lambda}.
\end{align*}
This finishes the proof.
\end{proof}

We are now ready to carry out the aforementioned dichotomy argument. The following proposition shows that if the constant in \eqref{MTWbound} is sufficiently small (i.e., the cost function is ``sufficiently close to satisfying \eqref{A3}'') then the operator norm of $W(x, t)$ either has a uniform upper bound independent of $\tmax$, or must be larger than a specific value somewhere; both of these threshold values are explicit. 

\begin{Proposition}\label{prop : dichotomy}
There exists a constant $\tilde{\sigma}_0>0$ depending only on the structure of the problem and $\Norm{u_{\text{initial}}}_{C^2(\overline{\O})}$ such that, as long as $0 < \MTWconst<\tilde{\sigma}_0$, exactly one of the two following alternatives hold: for any $0 \leq T < \tmax$, either
\begin{enumerate}[(i)]
 \item $$\max_{(x, t) \in \overline{\O} \times [0,T]} ||W(x, t)|| \geq {\frac{1}{n}}\left(\frac{1}{n\MTWconst}\right)^{\frac{1}{n-1}}, \quad \text{or }$$
 \item $$\max_{(x, t)  \in \overline{\O} \times [0,T]} ||W(x, t)|| \leq \frac{1}{2{ n}}\left(\frac{1}{n\MTWconst}\right)^{\frac{1}{n-1}}.$$
\end{enumerate}
\end{Proposition}

\begin{proof}
Define the auxiliary function $v: \Omega\times [0, \tmax)\times \S^{n-1}\to \R$ by
\begin{align*}
 v(x, t, \xi):=\log w_{\xi\xi}(x, t)+a\norm{\nabla u(x, t)}+\tilde Ca\psi(x, t),
\end{align*}
for constants $a, \tilde C>0$ to be determined, where $\psi$ is the function from Lemma \ref{lem: barrier}. Suppose $v$ achieves a maximum at $(x_0, t_0, \xi_0)$. Let us first assume $x_0$ belongs to the interior of $\O$. By following the calculations in \cite[Theorem 10.1]{Kitagawa12}, we obtain (note the second displayed block of equations on p.148 of \cite[Theorem 10.1]{Kitagawa12} is missing a term $2au_lw^{ij}D_{x_l}A_{ij}$, which yields the term $-\tilde{\tilde{C}}a\tr(W^{-1})$ below)
\begin{align}
0&\geq \mathcal{L}v(x_0, t_0, \xi_0) \notag \\
& \geq \frac{w^{ij}(D_{p_kp_\ell}A_{ij})w_{k\xi_0}w_{l\xi_0}}{w_{\xi_0\xi_0}}+(2a-C_1)\tr(W)-C_2a-(C_3+\tilde{\tilde{C}}a-\tilde Ca)\tr(W^{-1})\notag\\
&\geq \frac{w^{ij}(D_{p_kp_\ell}A_{ij})w_{k\xi_0}w_{l\xi_0}}{w_{\xi_0\xi_0}}+(2a-C_1)\tr(W)-C_2a \notag \\
& \label{eqn: L bound}
\end{align}
where the last inequality is obtained by taking $\tilde{C}$ large enough that $C_3+\tilde{\tilde{C}}a-\tilde Ca\leq 0$, and we assume $2a - C_1 > 0$. Here, it can be seen that the constants $C_1$, $C_2$, and $C_3$ depend on upper bounds on $\nabla u$, and on up to second order derivatives of $A_{ij}$ and $\log B$. From the definitions of $A_{ij}$ and $B$, and Proposition \ref{prop: nonMTW estimates} it can thus be seen that these constants along with $\tilde C$ and $a$ can be chosen to have upper bounds depending only on the structure of the problem.

Now diagonalize $(w_{ij}(x_0, t_0))$ and let $\lambda_1\leq \ldots \leq \lambda_n$ be the associated eigenvalues; note that, by choice, $\xi_0$ is an eigenvector corresponding to $\lambda_n$. We then estimate,
\begin{align}\label{eqn: trace inverse estimate}
 \tr(W^{-1})&=\sum_{i=1}^n \lambda_i^{-1}=\frac{\sum_{i=1}^n\prod_{j\neq i}\lambda_j}{\lambda_1\cdots\lambda_n}\leq \frac{n\lambda_n^{n-1}}{\lambda_1\cdots\lambda_n}\leq \frac{n \tr(W)^{n-1}}{\det W}.
\end{align}
Taking $V=\lambda_i^{-1/2}e_i$ for $1\leq i\leq n-1$ and $\eta=e_n$ in \eqref{MTWbound} and using \eqref{eqn: trace inverse estimate} yields
\begin{align*}
 \frac{w^{ij}(D_{p_kp_\ell}A_{ij})w_{k\xi_0}w_{l\xi_0}}{w_{\xi_0\xi_0}}&=\lambda_n\sum_{i=1}^n \lambda_i^{-1} D_{p_np_n}A_{ii}\\
& \geq -\MTWconst \lambda_n\sum_{i=1}^{n-1} \lambda_i^{-1}+D^2_{p_np_n}A_{nn}\\
& \geq -\MTWconst \lambda_n\sum_{i=1}^{n} \lambda_i^{-1}-\Norm{D^2_pA}\\
&\geq-\frac{n\MTWconst\lambda_n \tr(W)^{n-1}}{\det W}-\Norm{D^2_pA}\\
& \geq -C_4\MTWconst\tr(W)^n-\Norm{D^2_pA}
\end{align*}
where $C_4$ depends on $u_0$ and the structure of the problem. Combining with \eqref{eqn: L bound} and rearranging yields
\begin{align*}
 \tr(W)-\frac{C_4\MTWconst}{2a-C_1}\tr(W)^n & \leq  \frac{C_2a+\Norm{D^2_pA}}{2a-C_1}.
\end{align*}
Taking $a$ large enough that $\frac{C_4}{2a-C_1}\leq 1$, we then obtain
\begin{align}
\tr(W) - \MTWconst\tr(W)^n - C_5 \leq 0
\end{align}
for $C_5>0$ depending only on the structure of the problem.

Let us now assume $\MTWconst \in (0,\sigma_0]$, where $\sigma_0 > 0$ is a fixed constant satisfying 
\begin{equation}\label{definitionofsigmazero}
\sigma_0 \leq K_0 \quad \text{and} \quad \frac{1}{n\sigma_0} \geq \left(\frac{2 n{^2} C_5}{n - 1} \right)^{n - 1},
\end{equation}
where $K_0$ is the constant appearing in Lemma \ref{lem: barrier}. Examining the roots of the polynomial $s\mapsto s- \MTWconst s^n-C_5$ on the half-line $\{s \geq 0\}$ (see \eqref{goodinterval} and \eqref{badinterval} in Appendix \ref{polyprop}) yields two possibilities.

{\bf Case 1:} $\tr(W(x_0, t_0))\geq \left(\frac{1}{n\MTWconst}\right)^{\frac{1}{n-1}}.$ 

In this case, we find that 
\begin{align*}
\left(\frac{1}{n\MTWconst}\right)^{\frac{1}{n-1}} \leq \tr(W(x_0, t_0))\leq n\sup_{\overline{\O} \times [0,T]} \Norm{W}.
\end{align*}

{\bf Case 2:} $\tr(W(x_0, t_0))\leq \frac{1}{2{ n}}\left(\frac{1}{n\sigma_0}\right)^{\frac{1}{n-1}}.$ 

In this case, for any $(x, t, \xi)$ we obtain (recall $\xi_0$ is an eigenvector of $(w_{ij}(x_0,t_0))$ corresponding to $\lambda_n$)
\begin{align*}
 \log w_{\xi\xi}(x, t)&\leq  v(x_0, t_0, \xi_0) \\
 & \leq \log\left(\frac{1}{2{ n}}\left(\frac{1}{n\sigma_0}\right)^{\frac{1}{n-1}}\right)+a\norm{\nabla u(x_0, t_0)}+\tilde{C}a\psi(x_0, t_0)\\
 &\leq C_6+\log\left(\frac{1}{2{ n}}\left(\frac{1}{n\sigma_0}\right)^{\frac{1}{n-1}}\right)
\end{align*}
Taking exponentials and then a supremum over $(x, t, \xi)$ yields
\begin{align*}
 \sup_{\overline{\O} \times [0,T]} \Norm{W}\leq \frac{e^{C_6}}{2{ n}} \left(\frac{1}{n\sigma_0}\right)^{\frac{1}{n-1}} \leq \frac{1}{2{ n}}\left(\frac{1}{n\MTWconst}\right)^{\frac{1}{n-1}},
\end{align*}
so long as we choose $\MTWconst < \frac{1}{e^{(n-1)C_6}}\sigma_0$. 

Let us now assume $v$ achieves a maximum at $(x_0, t_0, \xi_0)$ where $x_0 \in \partial \O$. Then for any $(x,t,\xi)$, we have
\begin{align*}
 \log w_{\xi\xi}(x, t) &\leq  v(x_0, t_0, \xi_0)  = \log w_{\xi_0\xi_0}(x_0, t_0)+a\norm{\nabla u(x_0, t_0)}+\tilde Ca\psi(x_0, t_0) \\
 & \leq C_6 +\log\left( \sup_{\partial \O \times [0,T]} \Norm{W} \right)
\end{align*}
Consequently, there exists a constant $C_7 > 0$ depending only on the structure of the problem such that
\begin{equation}\label{bdyC2bound}
\sup_{\overline{\O} \times [0,T]} \Norm{W} \leq C_7  \sup_{\partial \O \times [0,T]} \Norm{W}.
\end{equation}
Upon inspection of \cite[Section 11]{Kitagawa12}, we find that the bound \eqref{bdyC2bound} allows us to show that 
$$\sup_{\overline{\O} \times [0,T]} \Norm{W} \leq C_8,$$
where $C_8$ depends only on the structure of the problem and the $C^2$ norm of $u_{\text{initial}}$ \footnote{The dependency of $C_8$ on $u_{\text{initial}}$ is through the constant $\alpha$ used in  \cite[Equation (9.9)]{Kitagawa12}. By the final displayed inequality on  \cite[pg. 143]{Kitagawa12}, it can be seen that $\alpha$ depends only on $||u_{\text{initial}}||_{C^2(\overline{\O})}$.}.

Choosing $\tilde{\sigma}_0 > 0$ to be a constant satisfying the inequalities
$$\tilde{\sigma}_0 \leq \frac{1}{e^{(n-1)C_6}}\sigma_0\quad \text{ and } \quad C_8 \leq \frac{1}{2{ n}}\left(\frac{1}{n\tilde{\sigma}_0}\right)^{\frac{1}{n-1}},$$
(and recalling that $\sigma_0$ already satisfies \eqref{definitionofsigmazero}), we obtain the conclusion of the Proposition.

\end{proof}
Before concluding this section, we note the following consequence of Proposition \ref{prop : dichotomy}. 

\begin{Corollary}\label{cor : initialsecondderivboundimpliesboundforalltime}
Suppose \eqref{MTWbound} is satisfied with $0 < \MTWconst \leq \tilde{\sigma}_0$, where $\tilde{\sigma}_0$ is the constant appearing in Proposition \ref{prop : dichotomy}. For $\tau \in [0, \tmax)$, consider the quantity 
$$\omega(\tau):= \max_{(x,t)  \in \overline{\O} \times [0,\tau]} ||W^{c,u}(x, t)||.$$
If $\omega(0) \leq \frac{1}{2{ n}} \left(\frac{1}{n\MTWconst}\right)^{\frac{1}{n-1}}$, then $\omega(\tau) \leq  \frac{1}{2{ n}} \left(\frac{1}{n\MTWconst}\right)^{\frac{1}{n-1}}$ for all $\tau \in [0, \tmax)$.
\end{Corollary}

\begin{proof}
This follows from the compactness of $\overline{\O}$ coupled with the joint continuity of $A^c,\nabla u, D^2u$ (and consequently of $W^{c,u}$) in $(x,t)$. Indeed, since $\omega(\tau)$ is an increasing function of $\tau$, we may suppose by contradiction that there exists $\tau^* \in [0,\tmax)$ such that $\omega(\tau^*) >  \frac{1}{2{ n}} \left(\frac{1}{n\MTWconst}\right)^{\frac{1}{n-1}}$ but $\omega(\tau) \leq  \frac{1}{2{ n}} \left(\frac{1}{n\MTWconst}\right)^{\frac{1}{n-1}}$ for all $0 \leq \tau < \tau^*$. By the assumption on $\omega(0)$, we know that $\tau^* > 0$. Then there must exist $x^* \in \overline{\O}$ such that $W^{c,u}(x^*,\tau^*)  = \omega(\tau^*)$. The dichotomy imposed by Proposition \ref{prop : dichotomy} implies  $W^{c,u}(x^*,\tau^*) \geq{ \frac{1}{n}} \left(\frac{1}{n\MTWconst}\right)^{\frac{1}{n-1}}$. On the other hand, since $W^{c,u}(x,t)$ is jointly continuous in $(x,t)$ whenever $(x,t) \in \O \times (0,\tmax)$, it follows that 
$${\frac{1}{n}} \left(\frac{1}{n\MTWconst}\right)^{\frac{1}{n-1}} \leq W^{c,u}(x^*,\tau^*) = \lim\limits_{\epsilon \to 0^+}  W^{c,u}(x^*,\tau^*-\epsilon) \leq \frac{1}{2{ n}} \left(\frac{1}{n\MTWconst}\right)^{\frac{1}{n-1}},$$
which is a contradiction.
\end{proof}

\section{Proof of Theorem \ref{thm : main}}\label{sec : mainthm}

We begin by letting $u_0$ be the unique solution of the steady state problem
\begin{equation}\label{staticPDE}
\begin{cases}
\det (W^{c_0, u_0}(x)) = B^{c_0}(x,\nabla u_0(x)), & \qquad x \in \O, \\
\G^{c_0}(x,\nabla u_0(x)) = 0, & \qquad x \in \partial \O, \\
\int_{\O} u_0(x) \ dx = 0.
\end{cases}
\end{equation}
Since $c_0$ satisfies the  \eqref{A3} condition, the results of \cite{TrudingerWang09} imply $u_0 \in C^{2,\a}(\overline \O)$, is strictly $c_0$-convex, and is locally, uniformly $c_0$-convex.

We assume $||c-c_0||_{C^4(\O \times \mathcal{N}_{r_0}(\O^*))} < \hat{R}$ for some $\hat{R}$ to be determined. We first choose $\hat{R} \leq R_0$ in Lemma \ref{lem: bitwist preserved} (1)-(3). This ensures $c$ satisfies the necessary structural conditions in Subsection \ref{subsec:notation}. Next, we  choose $\hat{R} \leq R_1$ in Lemma \ref{lem : initial}; this ensures the map $\Psi: B^{C^4(\overline \O\times \overline{\mathcal{N}_{r_0}(\O^*)})}_{R_1}(c_0)\to C^{2, \alpha}(\overline \O)$ is well-defined. Recalling that $\Psi(c_0) = u_0$, we next let $R_3 > 0$ be such that if $\hat{R} \leq R_3$, then $||\Psi(c)||_{C^{2,\a}(\overline{\O})} \leq 2||u_0||_{C^{2,\a}(\overline{\O})}$ and $\max_{\overline{\O}} ||W^{c,\Psi(c)}|| \leq 2\max_{\overline{\O}}||W^{c_0,u_0}||$; this follows from the continuity of the map $\Psi$ at $c = c_0$ and the compactness of $\overline{\O}$. Now let $u$ be the solution of \eqref{PDE} with $u_{\text{initial}} = \Psi(c)$. Then $u$ satisfies the estimates in Proposition \ref{prop: nonMTW estimates} (which do not require the \eqref{A3} condition) with constants $K_1, K_2$ depending only on $c_0$ and $u_0$. Finally, we let $R_4 > 0$ be such that if $\hat{R} \leq R_4$, then $c$ satisfies \eqref{MTWbound} with a constant $\MTWconst \leq \tilde{\sigma}_0$, where $\tilde{\sigma}_0$ is the constant appearing in Proposition \ref{prop : dichotomy}, and  $\max_{\overline{\O}}||W^{c_0,u_0}|| \leq  \frac{1}{4n} \left(\frac{1}{n\MTWconst}\right)^{\frac{1}{n-1}}$. Then Corollary \ref{cor : initialsecondderivboundimpliesboundforalltime} implies $\sup_{\overline{\O} \times [0,\tmax(c))} \Norm{W^{c,u}} \leq \frac{1}{2{ n}} \left(\frac{1}{n\MTWconst}\right)^{\frac{1}{n-1}}$. Consequently, $u$ satisfies uniform second derivative estimates independent of $\tmax(c)$. We can thus invoke the arguments from \cite[Section 12]{Kitagawa12} to conclude the infinite-time existence of the solution $u$ of \eqref{PDE} and the convergence to a potential function for the optimal transport problem between $(\O, \rho)$ and $(\O^*, \rho^*)$ with cost function $c$.

\appendix

\section{Properties of an $n$-th degree polynomial}\label{polyprop}

Fix $C \geq 0$ and consider the polynomial 
$$p_{\sigma}(s) := s - \sigma s^{n} - C \quad \text{for any } \sigma \in (0,\sigma_0],$$
where $\sigma_0 > 0$ is a constant satisfying the inequality
\begin{equation}\label{assumptiononepsilon}
\frac{1}{n\sigma_0} \geq \left(\frac{2 n{ ^2} C}{n - 1} \right)^{n - 1}.
\end{equation}
On the half-line $\{s \geq 0\}$, $p_{\sigma}$ has a single critical point at $s = \hat{s}(\sigma)$ where
\begin{equation}\label{criticalpoint}
\hat{s}(\sigma)^{n - 1} =  \frac{1}{n\sigma}.
\end{equation}
Since $\sigma \leq \sigma_0$, \eqref{assumptiononepsilon} implies
$$\hat{s}(\sigma)^{n - 1} =  \frac{1}{n\sigma} \geq \frac{1}{n \sigma_0} \geq \left(\frac{2 n{ ^2}  C}{n - 1} \right)^{n - 1}.$$
Consequently, $\hat{s}(\sigma) \geq \frac{2 n{ ^2}  C}{n - 1}$. This, combined with \eqref{criticalpoint}, implies
$$p_{\sigma}(\hat{s}(\sigma)) = \hat{s}(\sigma)( 1 - \sigma \hat{s}(\sigma)^{n - 1}) -C = \hat{s}(\sigma)\left(\frac{n - 1}{n} \right) - C \geq (2n-1) C > 0$$
whenever $\sigma \in (0, \sigma_0]$. Since $p_{\sigma}(0) = -C < 0$, it follows that there exists a root $s_1(\sigma) \in (0, \hat{s}(\sigma))$ of $p_{\sigma}$. Similarly, since $\lim\limits_{s \to +\infty} p_{\sigma}(s) = -\infty$ whenever $\sigma > 0$, there exists a second root $s_2(\sigma) \in (\hat{s}(\sigma), \infty)$ of $p_{\sigma}$. By Descartes' rules of signs, these are the only two positive, real roots of $p_\sigma$.

Clearly, $s_1(\sigma) = \sigma s_1(\sigma)^n + C \geq C$ for all $\sigma \in (0, \sigma_0]$. Also, since $s_1(\sigma) < \hat{s}(\sigma)$, we have by \eqref{criticalpoint}
$$s_1(\sigma)^{n - 1} < \hat{s}(\sigma)^{n - 1} = \frac{1}{n \sigma}.$$
This implies
$$1 - \sigma s_1(\sigma)^{n - 1}  > \frac{n - 1}{n}.$$
Since $p_{\sigma}(s_1(\sigma)) = 0$, we have $s_1(\sigma)(1 - \sigma s_1(\sigma)^{n - 1}) = C$, and so by \eqref{assumptiononepsilon}
$$
s_1(\sigma) = \frac{C}{1 - \sigma s_1(\sigma)^{n-1}}  < \frac{n C}{n - 1} \leq \frac{1}{2{ n}}\left(\frac{1}{n\sigma_0}\right)^{\frac{1}{n-1}}\quad \text{for all } \sigma \in (0, \sigma_0].
$$

On the other hand, since $s_2(\sigma) > \hat{s}(\sigma)$, we have
$$s_2(\sigma)^{n-1} > \hat{s}(\sigma)^{n - 1} = \frac{1}{n \sigma}.$$
Thus, 
$$s_2(\sigma) \geq \left(\frac{1}{n\sigma}\right)^{\frac{1}{n-1}} \quad \text{for all } \sigma \in (0,\sigma_0].$$
We conclude that if $p_{\sigma}(s) \leq 0$ for some $s \geq 0$, then either
\begin{equation}\label{goodinterval}
0 \leq s \leq \frac{1}{2{ n}}\left(\frac{1}{n\sigma_0}\right)^{\frac{1}{n-1}}, \text{ or }
\end{equation}
\begin{equation}\label{badinterval}
s \geq \left(\frac{1}{n\sigma}\right)^{\frac{1}{n-1}}.
\end{equation}

\bibliography{parabolic}
\bibliographystyle{amsplain}

\end{document}